\documentclass[11pt]{article}

\usepackage[T1]{fontenc}
\usepackage{lmodern}
\usepackage{microtype}
\usepackage{amsmath,amssymb,amsthm,mathtools}
\usepackage{bm}
\usepackage[a4paper,margin=30mm]{geometry}
\usepackage{enumitem}
\usepackage{xcolor}
\usepackage[colorlinks=true,linkcolor=blue!55!black,citecolor=blue!55!black,urlcolor=blue!55!black]{hyperref}

\hypersetup{
  pdftitle={Domain dependence of Sobolev--curl constants and structure of critical curl--curl solutions},
  pdfauthor={Yong Liu, Jun Wang, Xuan Wang, Wen Yang},
  pdfsubject={Critical curl--curl equations and intrinsic Sobolev--curl constants},
  pdfkeywords={Sobolev--curl inequality, thin domain, critical curl--curl equation, Newtonian potential, Maxwell boundary condition}
}

\allowdisplaybreaks
\setlist{itemsep=2pt,topsep=4pt}
\numberwithin{equation}{section}

\newtheorem{maintheorem}{Theorem}
\newtheorem{proposition}{Proposition}[section]
\newtheorem{lemma}[proposition]{Lemma}

\theoremstyle{definition}

\theoremstyle{remark}
\newtheorem{remark}[proposition]{Remark}

\newcommand{\R}{\mathbb{R}}
\newcommand{\G}{\mathcal{G}}
\newcommand{\dive}{\operatorname{div}}

\newcommand{\eps}{\varepsilon}
\newcommand{\dd}{\,\mathrm{d}}
\newcommand{\norm}[1]{\lVert #1\rVert}
\newcommand{\abs}[1]{\lvert #1\rvert}
\newcommand{\curl}{\operatorname{curl}}
\newcommand{\Scurl}{S_{\mathrm{curl}}}
\newcommand{\Sint}{S_{\mathrm{curl}}^{\mathrm{int}}}
\newcommand{\Sext}{S_{\mathrm{curl}}^{\mathrm{ext}}}
\newcommand{\Xzero}{X_{2,0}}
\newcommand{\Xzerop}{X_{p,0}}
\newcommand{\Wker}{\mathcal{W}}

\newcommand{\Needspace}[1]{%
  \par\begingroup
  \dimen0=#1\relax
  \dimen2=\pagegoal
  \advance\dimen2 by -\pagetotal
  \ifdim\dimen0>\dimen2\newpage\fi
  \endgroup
}

\title{Domain dependence of Sobolev--curl constants\\
and structure of critical curl--curl solutions}
\author{%
Yong Liu\thanks{Yong Liu, School of Mathematics and Statistics, Beijing Technology and Business University, Beijing, China. \textit{Email:} \href{mailto:yliumath@btbu.edu.cn}{\nolinkurl{yliumath@btbu.edu.cn}}.},\ \
Jun Wang\thanks{Jun Wang, School of Mathematical Sciences, Jiangsu University, Zhenjiang, Jiangsu, 212013, P.R. China. \textit{Email:} \href{mailto:wangmath2011@126.com}{\nolinkurl{wangmath2011@126.com}}.},\ \ 
Xuan Wang\thanks{Xuan Wang, Department of Mathematics, Faculty of Science and Technology, University of Macau, Macau, P.R. China. \textit{Email:} \href{mailto:wangxuanolivia@126.com}{\nolinkurl{wangxuanolivia@126.com}}.},\ \
Wen Yang\thanks{Wen Yang, Department of Mathematics, Faculty of Science and Technology, University of Macau, Macau, P.R. China. \textit{Email:} \href{mailto:wyang@um.edu.mo}{\nolinkurl{wyang@um.edu.mo}}.}%
}
\date{}

\begin{document}

\maketitle

\begin{abstract}
This paper investigates intrinsic critical Sobolev--curl constants on
bounded domains and the structure of finite-energy solutions to the
critical curl--curl equation in $\R^3$. First, for $1<p<3$, we prove that
the intrinsic critical $p$-curl constant on $(0,a)\times D$ is bounded
above by $a^{p/3}$ times a Sobolev quotient on $D$. The proof uses an
aligned test field whose optimal curl-free correction vanishes.
Consequently, these constants tend to zero along thin convex boxes.
For $p=2$, we construct explicit convex and nonconvex domains with
intrinsic constant below the sharp scalar Sobolev constant. The
nonconvex example gives an affirmative answer to the domain-existence
question in Problem~5 of Mederski and Szulkin(Arch. Ration. Mech. Anal. 241 (2021), 1815-1842), while the convex
example disproves the lower bound stated there for convex domains.
Concerning their Problem~2, we obtain an implicit Newtonian
representation of finite-energy solutions and exclude two special forms
of solutions.
These structural results do not yet provide an explicit ground state
or an unconditional ground-state decay rate as asked in Problem~2.
\end{abstract}



\section{Introduction}\label{sec:introduction}

The study of critical elliptic equations is closely related to sharp
Sobolev inequalities. In dimension three, a basic example is
\begin{equation}\label{eq:scalar-critical-intro}
 -\Delta U=\abs{U}^{4}U
 \qquad\text{in }\R^3.
\end{equation}
This equation is associated with the sharp Sobolev inequality
\[
 S\left(\int_{\R^3}\abs{U}^{6}\dd x\right)^{1/3}
 \leq \int_{\R^3}\abs{\nabla U}^{2}\dd x.
\]
Aubin \cite{Aubin1976} and Talenti \cite{Talenti1976} determined the
optimal constant and its extremals. The classification of positive
entire solutions is closely related to the symmetry theorem of
Caffarelli, Gidas and Spruck \cite{CaffarelliGidasSpruck1989}.
In particular, the positive solutions of \eqref{eq:scalar-critical-intro}
are the Aubin--Talenti profiles, up to translations and dilations.
These symmetries also account for the loss of compactness at the
critical exponent. Lions \cite{Lions1985} developed the
concentration--compactness principle in the limit case to analyze
this phenomenon, including its role in sharp functional inequalities.

On bounded domains, the critical exponent remains a source of
noncompactness even though escape by translation is no longer possible.
The work of Brezis and Nirenberg \cite{BrezisNirenberg1983} on critical
Dirichlet problems shows how a lower-order perturbation can restore
existence in appropriate parameter ranges. In this setting, the sharp
Sobolev constant supplies the natural comparison level for variational
arguments. For vector equations involving the curl operator, the
analogous questions also depend on the kernel of the operator and
the boundary condition imposed on the field.

In this paper, we consider the critical curl--curl equation
\begin{equation}\label{eq:main}
 \curl(\curl u)=\abs{u}^{4}u
 \qquad\text{in }\R^3,
\end{equation}
together with the associated Sobolev--curl constants on bounded
domains. Here $\curl u=\nabla\times u$ and $\dive u=\nabla\cdot u$.
All derivatives are understood in the distributional sense when
necessary. Equation~\eqref{eq:main} arises in the study of nonlinear
Maxwell equations and has the same critical scaling as
\eqref{eq:scalar-critical-intro}. By the identity
\[
 \curl(\curl u)=\nabla(\dive u)-\Delta u,
\]
the curl--curl operator agrees with $-\Delta$ on divergence-free
fields. However, every gradient field belongs to its kernel.
The resulting infinite-dimensional kernel is a main difficulty
in extending scalar variational arguments to \eqref{eq:main}.

More generally, time-harmonic models for nonlinear electromagnetic
media lead to equations of the form
\[
 \curl\bigl(\mu(x)^{-1}\curl u\bigr)
 -\omega^2\varepsilon(x)u=f(x,u),
\]
where $\mu$ and $\varepsilon$ describe the magnetic permeability and
linear permittivity, and $f$ accounts for the nonlinear polarization.
For the modeling assumptions and variational formulation, we refer to
Bartsch and Mederski \cite{BartschMederski2017}. Equation~\eqref{eq:main}
is the homogeneous critical model with no linear zeroth-order term.
Its analysis retains both the degeneracy of the curl--curl operator
and the loss of compactness associated with critical growth.

An early variational treatment of the degeneracy was given by
Benci and Fortunato \cite{BenciFortunato2004}, using a decomposition
into solenoidal and gradient components. For bounded domains,
Bartsch and Mederski \cite{BartschMederski2015} obtained ground and
bound state results for semilinear time-harmonic Maxwell equations.
Mederski \cite{Mederski2018} subsequently studied the
Brezis--Nirenberg problem for the curl--curl operator with the
perfect-conductor boundary condition $\nu\times u=0$, obtaining
critical-growth results in a cylindrically symmetric setting.

In the whole space, Mederski \cite{Mederski2015} proved ground-state
existence for a class of semilinear Maxwell equations with vanishing
permittivity. Bartsch, Dohnal, Plum and Reichel
\cite{BartschDohnalPlumReichel2016} studied ground states in
cylindrically symmetric media. Without imposing this symmetry,
Mederski, Schino and Szulkin \cite{MederskiSchinoSzulkin2020}
established existence and multiplicity results under suitable
growth assumptions formulated in Orlicz spaces. The generalized
Nehari approach and the difficulties caused by the infinite-dimensional
kernel are discussed further in \cite{MederskiSchino2022}.

For the purely critical equation
\eqref{eq:main}, Mederski and Szulkin \cite{MederskiSzulkin2021}
established a sharp Sobolev--curl inequality and proved that its
optimal constant is attained. After a curl-free correction and an
amplitude normalization, an optimizer gives a ground state of
\eqref{eq:main}. The corresponding critical $p$-curl inequalities
are studied in \cite{MederskiSzulkin2025}.

Boundary conditions are particularly important when passing from
the whole space to a bounded domain. The condition $\nu\times u=0$
prescribes only the tangential component and is weaker than requiring
every component to have zero trace. Vector potentials and decompositions
on nonsmooth domains are treated by Amrouche, Bernardi, Dauge and
Girault \cite{AmroucheBernardiDaugeGirault1998}. We also refer to
Monk \cite{Monk2003} for the Maxwell function spaces and traces.
These distinctions matter here because the admissible curl-free
corrections enter the definition of the optimal constant itself.

For a bounded domain $\Omega$, the choice of curl-free corrections
leads to two different constants. In the intrinsic constant
$\Sint(\Omega)$, the correction satisfies the Maxwell boundary
condition on $\Omega$. In the constant $\Sext(\Omega)$, the field
is extended by zero and the correction is taken in the whole space.
Their precise definitions are given in
Sections~\ref{sec:main-results} and~\ref{sec:examples}. For $p=2$,
our notation agrees with that of \cite{MederskiSzulkin2021} as follows:
\[
 \Sint(\Omega)=\overline S_{\mathrm{curl}}(\Omega),
 \qquad
 \Sext(\Omega)=S_{\mathrm{curl}}(\Omega)
 \quad\text{in \cite{MederskiSzulkin2021}}.
\]
Mederski and Szulkin proved that $\Sext(\Omega)=\Scurl>S$, where
$S$ is the sharp scalar Sobolev constant. They also stated the
bound $\overline S_{\mathrm{curl}}(\Omega)\geq S$ for bounded
convex domains or domains with $C^{1,1}$ boundary.

Two questions in \cite[Section~7]{MederskiSzulkin2021} motivate
the present work. We restate them below, using the numbering in that paper.
\begin{enumerate}[label={},leftmargin=0pt,itemindent=0pt,
                  listparindent=\parindent,itemsep=8pt]
\item \textbf{Problem 2.}
Find an explicit formula for a ground state of \eqref{eq:main},
or determine its decay at infinity. In particular, how does this
decay compare with that of the scalar Aubin--Talenti solutions?
\item \textbf{Problem 5.}
Determine whether the stated comparison
\[
 \Scurl\geq\overline S_{\mathrm{curl}}(\Omega)\geq S
\]
can be improved. Also, determine whether some bounded Lipschitz
domain $\Omega$, neither convex nor of class $C^{1,1}$, satisfies
\[
 \overline S_{\mathrm{curl}}(\Omega)<S.
\]
\end{enumerate}

Theorem~\ref{thm:examples-extension} answers the existence question
in Problem~5 by constructing a nonconvex Lipschitz domain with
$\Sint(\Omega)<S$ and boundary not of class $C^{1,1}$.
For Problem~2, Theorem~\ref{thm:main-conclusions} provides an
implicit reconstruction.
An explicit ground state and its general decay rate remain undetermined.

To answer the existence question in Problem~5, we first study the dependence of the
intrinsic constant on the geometry of the domain. For a fixed
bounded Lipschitz set $D\subset\R^2$, we show that the intrinsic
critical $p$-curl constant on $\Omega_a=(0,a)\times D$ is bounded
above by a multiple of $a^{p/3}$ for every $1<p<3$. In particular,
it tends to zero along thin convex boxes. For $p=2$, we give a
convex box and a nonconvex prism whose constants are strictly
smaller than $S$. The convex box also shows that the lower bound
stated in \cite{MederskiSzulkin2021} for
convex domains does not hold for the intrinsic constant.

To obtain these estimates, we use fields of the form
$U_f(s,y)=f(y)e_1$. The main point is to prove that their optimal
intrinsic curl-free correction is zero. For this purpose, we
first represent the admissible curl-free fields as gradients of
zero-trace potentials. We then verify the variational identity
for $U_f$ by integration in the axial variable. The resulting
quotient contains the factor $a^{p/3}$. This construction also
explains the obstruction to applying the sharp scalar Sobolev
inequality on $\Omega$: the Maxwell condition prescribes the
tangential trace, whereas the scalar zero-extension argument
requires zero trace of every component. We discuss this point
in Remark~\ref{rem:published-bound}. The whole-space identity
$\Sext(\Omega)=\Scurl>S$ remains valid.

We next study the structure of finite-energy solutions of
\eqref{eq:main}. Taking the divergence gives
\[
 \dive(\abs{u}^{4}u)=0.
\]
Thus the equation imposes a constraint on the nonlinear source,
rather than a divergence-free condition on $u$. Using the Helmholtz
decomposition and minimization along curl-free directions, we
express the solenoidal component as a Newtonian potential and
recover the remaining component through a unique nonlinear
correction.

We also exclude constant polarizations of Aubin--Talenti functions
and nonzero constant-coefficient Beltrami fields. The decay of general
ground states remains open in this analysis.

The paper is organized as follows. In Section~\ref{sec:main-results},
we introduce the functional setting and state the main results.
Section~\ref{sec:proof-one} is devoted to the proofs of
Theorems~\ref{thm:main-intro} and~\ref{thm:examples-extension}.
In Section~\ref{sec:proof-two}, we prove
Theorem~\ref{thm:main-conclusions}.

\section{Main results}\label{sec:main-results}

Let $\Omega\subset\R^3$ be a nonempty open set, let \(1<p<3\), and set
\[
 p^*=\frac{3p}{3-p}.
\]
Define
\[
 W^{p^*,p}(\curl;\Omega)
 :=\left\{u\in L^{p^*}(\Omega;\R^3):
             \curl u\in L^p(\Omega;\R^3)\right\},
\]
with norm
\[
 \|u\|_{W^{p^*,p}(\curl;\Omega)}
 :=\|u\|_{L^{p^*}(\Omega)}+\|\curl u\|_{L^p(\Omega)}.
\]
Set
\[
\begin{aligned}
 \Xzerop(\Omega):=\Bigl\{u\in W^{p^*,p}(\curl;\Omega):\;&
 \text{there exists }\{u_n\}\subset C_c^\infty(\Omega;\R^3)
 \text{ such that}\\[-1mm]
 &u_n\to u\text{ in }L^{p^*}(\Omega;\R^3),
 \quad \curl u_n\to\curl u\text{ in }L^p(\Omega;\R^3)
 \Bigr\},
\end{aligned}
\]
and
\[
 \qquad
 \Wker_{p,\Omega}:=
 \{w\in\Xzerop(\Omega):\curl w=0\}.
\]
The intrinsic critical \(p\)-curl constant is
\begin{equation}\label{eq:p-constant}
 S_{p,\mathrm{curl}}^{\mathrm{int}}(\Omega)
 :=\inf_{\substack{u\in\Xzerop(\Omega)\\\curl u\neq0}}
 \frac{\displaystyle\int_\Omega|\curl u|^p\,dx}
 {\left(
   \displaystyle\inf_{w\in\Wker_{p,\Omega}}
   \int_\Omega|u+w|^{p^*}\,dx
  \right)^{p/p^*}}.
\end{equation}

For a bounded Lipschitz domain $D\subset\R^2$, define
\begin{equation*}
 \Lambda_{p,p^*}(D):=
 \inf_{0\ne f\in W_0^{1,p}(D)}
 \frac{\displaystyle\int_D|\nabla f|^p\,dy}
 {\left(\displaystyle\int_D|f|^{p^*}\,dy\right)^{p/p^*}}.
\end{equation*}
This constant is finite and positive by the two-dimensional Sobolev
embedding $W_0^{1,p}(D)\hookrightarrow L^{p^*}(D)$.

\Needspace{7\baselineskip}
We first give an upper bound for the intrinsic constant on product
domains. As a consequence, there is no positive lower bound which
is uniform over all bounded convex Lipschitz domains.
\begin{maintheorem}\label{thm:main-intro}
Let $1<p<3$, $p^*=3p/(3-p)$, and let $D\subset\R^2$ be a fixed
bounded Lipschitz domain. Set $Q=(0,1)^2$ and
$\Omega_a=(0,a)\times D\subset\R^3$ for $a>0$.
The following assertions hold.
\begin{enumerate}[label=\textup{(\roman*)},leftmargin=2.2em]
\item \(\displaystyle
 S_{p,\mathrm{curl}}^{\mathrm{int}}(\Omega_a)
 \leq a^{p/3}\Lambda_{p,p^*}(D).
\)
\item \(\displaystyle
 \lim_{a\to0}S_{p,\mathrm{curl}}^{\mathrm{int}}
       \bigl((0,a)\times Q\bigr)=0
 \quad\text{and}\quad
 \inf_{\substack{\Omega\subset\R^3\\
                  \Omega\text{ bounded, convex}\\
                  \text{and Lipschitz}}}
       S_{p,\mathrm{curl}}^{\mathrm{int}}(\Omega)=0.
\)
\end{enumerate}
\end{maintheorem}

For Theorem~\ref{thm:examples-extension}, take $p=2$ in the preceding
definitions and abbreviate
\begin{equation}\label{eq:Sint}
 \Sint(\Omega):=S_{2,\mathrm{curl}}^{\mathrm{int}}(\Omega).
\end{equation}

The scalar comparison constant is
\[
 S=\inf_{0\neq\varphi\in C_c^\infty(\R^3)}
 \frac{\displaystyle\int_{\R^3}|\nabla\varphi|^2\,dx}
      {\left(\displaystyle\int_{\R^3}|\varphi|^6\,dx\right)^{1/3}}
 =3\left(\frac{\pi}{2}\right)^{4/3},
\]
whose sharp value and extremals are classical
\cite{Aubin1976,Talenti1976}.

\Needspace{7\baselineskip}
For $p=2$, the preceding estimate yields the following explicit
examples with intrinsic constant smaller than $S$. The nonconvex
example answers the domain-existence question in Problem~5.
\begin{maintheorem}
\label{thm:examples-extension}
Let $Q=(0,1)^2$. The two bounded Lipschitz domains
\[
 \Omega_B=(0,\tfrac14)\times Q,\qquad
 \Omega_L=(0,\tfrac14)\times D_L,
\]
where
\[
 D_L=\bigl((0,2)\times(0,1)\bigr)
      \cup\bigl((0,1)\times(0,2)\bigr),
\]
satisfy
\[
 \Sint(\Omega_B),\ \Sint(\Omega_L)
 \leq\pi^2\left(\frac{2}{25}\right)^{1/3}<S.
\]
$\Omega_B$ is convex, $\Omega_L$ is nonconvex, and
$\partial\Omega_B,\partial\Omega_L\notin C^{1,1}$.
\end{maintheorem}

\begin{remark}
\label{rem:problem-P5}
Theorem~\ref{thm:examples-extension} settles the domain-existence
question in \cite[Section~7, Problem~5]{MederskiSzulkin2021}.
Indeed, $\Omega_L$ has all the required properties:
\[
 \Omega_L\ \text{is bounded, Lipschitz and nonconvex},\qquad
 \partial\Omega_L\notin C^{1,1},
\]
and
\[
 \overline S_{\mathrm{curl}}(\Omega_L)
 =\Sint(\Omega_L)
 \leq\pi^2\left(\frac{2}{25}\right)^{1/3}<S.
\]
For the comparison inequality also mentioned in Problem~5, the convex
example $\Omega_B$ disproves the stated lower bound, and
Theorem~\ref{thm:main-intro} excludes any positive uniform
intrinsic lower bound over bounded convex Lipschitz domains.
The domain-independent result $\Sext(\Omega)=\Scurl>S$ concerns
whole-space corrections and remains unchanged.
\end{remark}

Theorem~\ref{thm:main-intro} follows from the aligned field
$U_f(s,y)=f(y)e_1$, whose optimal intrinsic curl-free correction is zero.
The critical quotient therefore contains the factor $a^{1-p/p^*}=a^{p/3}$.
Taking $p=2$ and $a=1/4$ gives the two counterexamples in
Theorem~\ref{thm:examples-extension}. The convex example disproves the
intrinsic lower bound stated in
\cite[Theorem~1.3]{MederskiSzulkin2021}, without affecting the
domain-independent whole-space-corrected constant.

For the whole-space equation \eqref{eq:main}, take $p=2$ and
$\Omega=\R^3$ in the preceding definitions. The energy on
$W^{6,2}(\curl;\R^3)$ is
\begin{equation}\label{eq:energy}
    J(u)
    :=\frac12\int_{\R^{3}}\abs{\curl  u}^{2}\dd x
      -\frac16\int_{\R^{3}}\abs{u}^{6}\dd x,
\end{equation}
and a weak solution is a field $u\in W^{6,2}(\curl;\R^3)$ satisfying
\begin{equation}\label{eq:weak}
    \int_{\R^{3}}(\curl  u)\cdot(\curl \phi)\dd x
    =\int_{\R^{3}}\abs{u}^{4}u\cdot\phi\dd x
    \qquad
    \text{for every }\phi\in C_{c}^{\infty}(\R^{3},\R^{3}).
\end{equation}

The positive Aubin--Talenti functions in $\R^3$, normalized to solve
$-\Delta U=U^5$, are
\begin{equation}\label{eq:aubin-talenti}
 U_{\varepsilon,y}(x)
 =3^{1/4}\left(\frac{\varepsilon}{\varepsilon^2+|x-y|^2}\right)^{1/2},
 \qquad \varepsilon>0,\quad y\in\R^3.
\end{equation}
They attain the sharp scalar Sobolev constant $S$
\cite{Aubin1976,Talenti1976}.

\Needspace{7\baselineskip}
Our next result concerns finite-energy solutions of \eqref{eq:main}.
We give their Newtonian representation and exclude two special forms
of solutions.
\begin{maintheorem}\label{thm:main-conclusions}
Let $u\in W^{6,2}(\curl;\R^3)$ be a weak solution of \eqref{eq:main}. Set
\[
 F=|u|^4u,\qquad N(x)=\frac{1}{4\pi|x|}.
\]
\begin{enumerate}[label=\textup{(\roman*)},leftmargin=2.2em]
\item There are unique $v,g\in W^{6,2}(\curl;\R^3)$ such that
      \[
       u=v+g,\qquad \dive  v=0,\qquad \curl  g=0.
      \]
      Here $v=N*F$, and $g$ is the unique minimizer
      \[
       g=\operatorname*{argmin}_{\substack{h\in W^{6,2}(\curl;\R^3)\\
                          \curl  h=0}}
          \frac16\int_{\R^3}|v+h|^6\dd x.
      \]
\item A ground state cannot have the form $u(x)=aU(x)$, where
      $a\in\R^3\setminus\{0\}$ is constant and $U$ is an Aubin--Talenti
      function. No nonzero solution $u$ satisfies
      $\curl  u=\kappa u$ for a constant $\kappa\in\R$.
\end{enumerate}
\end{maintheorem}

\begin{remark}\label{rem:problem-P2}
Problem~2 in \cite[Section~7]{MederskiSzulkin2021} asks for
an explicit ground-state expression or information on ground-state
decay. Theorem~\ref{thm:main-conclusions} provides structural
results relevant to this question, but does not resolve it.
The representation in part~(i) remains implicit because its source
$F=|u|^4u$ depends on the unknown solution. Ground-state decay
requires further estimates for the source and the nonlinear correction.
\end{remark}

\section{Proofs of Theorems 1 and 2}\label{sec:proof-one}

In this section, we prove Theorems~\ref{thm:main-intro}
and~\ref{thm:examples-extension}. We first describe the curl-free
kernel and the associated minimization problem. We then construct
aligned test fields on product domains and compute their quotients.

\subsection{Curl-free corrections and the boundary condition}\label{sec:setting}

Throughout the proof of Theorem~\ref{thm:main-intro}, fix $1<p<3$ and
$p^*=3p/(3-p)$.

\Needspace{7\baselineskip}
We begin with a representation of the curl-free kernel. The
connectedness assumption below ensures that the potential has zero trace.
\begin{proposition}
\label{thm:potential-representation}
Let $\Omega\subset\R^3$ be a bounded Lipschitz domain with connected
exterior $\R^3\setminus\overline\Omega$. Then
\begin{equation}\label{eq:p-kernel}
 \Wker_{p,\Omega}=\nabla W_0^{1,p^*}(\Omega).
\end{equation}
Each $w\in\Wker_{p,\Omega}$ has a unique potential
$\phi_w\in W_0^{1,p^*}(\Omega)$, with
$\|\phi_w\|_{W^{1,p^*}(\Omega)}\leq C_{\Omega,p}\|w\|_{L^{p^*}(\Omega)}$.
\end{proposition}
\begin{proof}
Let $w\in\Wker_{p,\Omega}$. Choose
$w_n\in C_c^\infty(\Omega;\R^3)$ such that $w_n\to w$ in $L^{p^*}$
and $\curl w_n\to0$ in $L^p$. Their zero extensions satisfy
\[
 \widetilde w_n\to\widetilde w\quad\text{in }L^{p^*}(\R^3),
 \qquad
 \curl\widetilde w_n=\widetilde{\curl w_n}\to0
 \quad\text{in }L^p(\R^3).
\]
It follows that $\curl\widetilde w=0$ in distributions. By the
distributional Poincar\'e lemma, there exists
$\Phi\in W_{\mathrm{loc}}^{1,p^*}(\R^3)$ with
$\nabla\Phi=\widetilde w$.

Since $\nabla\Phi=0$ on the connected exterior
$\R^3\setminus\overline\Omega$, $\Phi$ is constant there.
Subtracting this constant, we may assume that $\Phi=0$ on the
exterior. The boundedness and Lipschitz regularity of $\Omega$
then imply that $\Phi\in W^{1,p^*}(\R^3)$ and
$\phi_w:=\Phi|_\Omega\in W_0^{1,p^*}(\Omega)$. Thus
\begin{equation}\label{eq:kernel-first-inclusion}
 \Wker_{p,\Omega}\subset\nabla W_0^{1,p^*}(\Omega).
\end{equation}

Conversely, if $\phi\in W_0^{1,p^*}(\Omega)$, take
$\phi_n\in C_c^\infty(\Omega)$ converging to $\phi$ in
$W^{1,p^*}(\Omega)$. Then $\nabla\phi_n\to\nabla\phi$ in $L^{p^*}$
and $\curl(\nabla\phi_n)=0$. By the definition of $\Xzerop(\Omega)$,
we obtain
\begin{equation}\label{eq:kernel-second-inclusion}
 \nabla W_0^{1,p^*}(\Omega)\subset\Wker_{p,\Omega}.
\end{equation}
This proves \eqref{eq:p-kernel}. Finally, Poincar\'e's inequality
gives uniqueness of the zero-trace potential and the estimate
\[
 \|\phi_w\|_{W^{1,p^*}(\Omega)}
 \leq C_{\Omega,p}\|\nabla\phi_w\|_{L^{p^*}(\Omega)}
 =C_{\Omega,p}\|w\|_{L^{p^*}(\Omega)}.
\]
\end{proof}

\begin{remark}\label{rem:topology-sharp}
If the exterior is disconnected, the potential may have different
constants on its components, and subtracting one constant need not give
zero boundary trace. The connected-exterior hypothesis is therefore
essential to the representation used here. At $p=2$, it gives
$\Wker_{2,\Omega}=\nabla W_0^{1,6}(\Omega)$.
\end{remark}

\Needspace{7\baselineskip}
To define the curl-free correction, we use the following consequence
of weak compactness and strict convexity.
\begin{lemma}
\label{lem:strict-convex-projection}
Let $E\subset\R^3$ be measurable, let \(1<p<\infty\), let \(Z\) be a closed linear subspace of
\(L^{p}(E,\R^{m})\), and fix \(q\in L^{p}\).  Then
\begin{equation*}
    \inf_{z\in Z}\int_E\abs{q+z}^{p}\dd x
\end{equation*}
has a unique minimizer \(z_{q}\in Z\).  It is characterized by
\begin{equation}\label{eq:abstract-Euler}
    \int_E
       \abs{q+z_{q}}^{p-2}(q+z_{q})\cdot h\dd x=0
    \qquad(h\in Z).
\end{equation}
Moreover,
\begin{equation*}
    \norm{q+z_{q}}_{p}\leq\norm{q}_{p},
    \qquad
    \norm{z_{q}}_{p}\leq2\norm{q}_{p}.
\end{equation*}
\end{lemma}

\begin{proof}
Set $\mathcal F(z)=\|q+z\|_p^p$ and
$\mu=\inf_{z\in Z}\mathcal F(z)$. Let $\{z_n\}\subset Z$ be a
minimizing sequence. Since
\[
 \|z_n\|_p\leq\|q+z_n\|_p+\|q\|_p,
\]
the sequence is bounded in $L^p$. After passing to a subsequence,
$z_n\rightharpoonup z_q$ in $L^p$. The space $Z$ is weakly closed,
and weak lower semicontinuity gives
\[
 z_q\in Z,\qquad
 \mathcal F(z_q)\leq\liminf_{n\to\infty}\mathcal F(z_n)=\mu.
\]
Hence $z_q$ is a minimizer. Its uniqueness follows from the strict
convexity of $\mathcal F$.

For $h\in Z$, differentiating $\mathcal F(z_q+th)$ at $t=0$
gives \eqref{eq:abstract-Euler}. The differentiation is justified
by H\"older's inequality. Conversely, if
\eqref{eq:abstract-Euler} holds, convexity gives, for every $z\in Z$,
\[
 \mathcal F(z)\geq\mathcal F(z_q)
 +p\int_E|q+z_q|^{p-2}(q+z_q)\cdot(z-z_q)\dd x
 =\mathcal F(z_q).
\]
Thus the Euler condition also characterizes the minimizer.
Comparing with $z=0$ yields $\|q+z_q\|_p\leq\|q\|_p$, and the
triangle inequality gives $\|z_q\|_p\leq2\|q\|_p$.
\end{proof}

\Needspace{7\baselineskip}
Applying Lemma~\ref{lem:strict-convex-projection} to the curl-free
kernel gives the correction used in \eqref{eq:p-constant}.
\begin{proposition}\label{prop:projection}
For every bounded Lipschitz domain $\Omega$ and $u\in\Xzerop(\Omega)$,
there is a unique $w_{p,\Omega}(u)\in\Wker_{p,\Omega}$ such that
\begin{equation}\label{eq:projection-minimum}
 \int_\Omega|u+w_{p,\Omega}(u)|^{p^*}\,dx
 =\min_{w\in\Wker_{p,\Omega}}\int_\Omega|u+w|^{p^*}\,dx.
\end{equation}
It is characterized by
\begin{equation}\label{eq:euler-correction}
 \int_\Omega |u+w_{p,\Omega}(u)|^{p^*-2}
 (u+w_{p,\Omega}(u))\cdot z\,dx=0\qquad(z\in\Wker_{p,\Omega}).
\end{equation}
For $p=2$, write $w_\Omega(u)=w_{2,\Omega}(u)$.
\end{proposition}
\begin{proof}
It suffices to prove that $\Wker_{p,\Omega}$ is closed in $L^{p^*}$.
Suppose that $w_n\in\Wker_{p,\Omega}$ and $w_n\to w$ in $L^{p^*}$.
Passing to distributions gives $\curl w=0$. Consequently,
\[
 \|w_n-w\|_{W^{p^*,p}(\curl;\Omega)}
 =\|w_n-w\|_{L^{p^*}(\Omega)}\to0.
\]
Since $\Xzerop(\Omega)$ is closed in the graph norm,
$w\in\Xzerop(\Omega)$ and hence $w\in\Wker_{p,\Omega}$.
The conclusion now follows from
Lemma~\ref{lem:strict-convex-projection}, applied with exponent $p^*$
and $Z=\Wker_{p,\Omega}$.
\end{proof}

\subsection{Aligned fields and proof of Theorem 1}\label{sec:aligned}
Let $D\subset\R^2$ be bounded and Lipschitz and set
$\Omega_a=(0,a)\times D$, with $x=(s,y)$.

\Needspace{7\baselineskip}
Before applying Proposition~\ref{thm:potential-representation}, we
verify that the exterior of $\Omega_a$ is connected.
\begin{lemma}
\label{lem:product-exterior}
For every bounded domain \(D\subset\mathbb R^2\) and every \(a>0\),
\(\mathbb R^3\setminus([0,a]\times\overline D)\) is path connected.
\end{lemma}
\begin{proof}
Choose $y_*\notin\overline D$ and put $x_*=(a+1,y_*)$.
For an exterior point $x=(s,y)$ with $s\notin[0,a]$, the path
\begin{equation}
 (s,y)\longrightarrow(s,y_*)\longrightarrow(a+1,y_*)
 \label{eq:exterior-path-one}
\end{equation}
lies in the exterior. Indeed, its first segment has axial coordinate
outside $[0,a]$, and its second has cross-sectional coordinate
$y_*\notin\overline D$.

If $s\in[0,a]$, then $y\notin\overline D$. In this case, use
\begin{equation}
 (s,y)\longrightarrow(a+1,y)\longrightarrow(a+1,y_*).
 \label{eq:exterior-path-two}
\end{equation}
Both segments again lie in the exterior. Every exterior point can
therefore be joined to $x_*$, which proves the assertion.
\end{proof}
The argument applies also when $D$ is multiply connected.

For $f\in W_0^{1,p}(D)$, define
\begin{equation*}
 U_f(s,y)=f(y)e_1=(f(y),0,0).
\end{equation*}
Then
\begin{equation}\label{eq:curl-aligned}
 \curl U_f=(0,\partial_{y_2}f,-\partial_{y_1}f),
 \qquad |\curl U_f|=|\nabla_y f|.
\end{equation}

\Needspace{7\baselineskip}
We first check that $U_f$ is admissible in \eqref{eq:p-constant}.
\begin{lemma}\label{lem:admissible}
For every $f\in W_0^{1,p}(D)$, $U_f\in\Xzerop(\Omega_a)$.
\end{lemma}
\begin{proof}
First let $f\in C_c^\infty(D)$. Choose
$\eta_n\in C_c^\infty(0,a)$ with $0\leq\eta_n\leq1$ and
$\eta_n\to1$ pointwise. Since $\nabla\eta_n$ is parallel to $U_f$,
\[
 \curl(\eta_nU_f)
 =\nabla\eta_n\times U_f+\eta_n\curl U_f
 =\eta_n\curl U_f.
\]
By dominated convergence,
\begin{align*}
 \|\eta_nU_f-U_f\|_{L^{p^*}(\Omega_a)}^{p^*}
 &=\int_0^a|\eta_n-1|^{p^*}\dd s\int_D|f|^{p^*}\dd y\to0,\\
 \|\curl(\eta_nU_f)-\curl U_f\|_{L^p(\Omega_a)}^p
 &=\int_0^a|\eta_n-1|^p\dd s\int_D|\nabla f|^p\dd y\to0.
\end{align*}
As $\eta_nU_f\in C_c^\infty(\Omega_a;\R^3)$, this proves
$U_f\in\Xzerop(\Omega_a)$ for smooth $f$.

For $f\in W_0^{1,p}(D)$, take $f_k\in C_c^\infty(D)$ converging
to $f$ in $W^{1,p}(D)$. The two-dimensional Sobolev embedding
also gives convergence in $L^{p^*}(D)$. Here
$p^*<2p/(2-p)$ if $p<2$, while every finite exponent is
admissible if $p\geq2$. It follows from \eqref{eq:curl-aligned} that
\begin{align*}
 \|U_{f_k}-U_f\|_{L^{p^*}(\Omega_a)}
 &=a^{1/p^*}\|f_k-f\|_{L^{p^*}(D)}\to0,\\
 \|\curl U_{f_k}-\curl U_f\|_{L^p(\Omega_a)}
 &=a^{1/p}\|\nabla f_k-\nabla f\|_{L^p(D)}\to0.
\end{align*}
The closedness of $\Xzerop(\Omega_a)$ completes the proof.
\end{proof}

\Needspace{7\baselineskip}
We next prove that the optimal curl-free correction of $U_f$ vanishes.
This allows us to compute the denominator in \eqref{eq:p-constant}.
\begin{lemma}\label{lem:zero-correction}
For every $f\in W_0^{1,p}(D)$, zero is the unique minimizer of
\begin{equation*}
 w\longmapsto\int_{\Omega_a}|U_f+w|^{p^*}\,dx,
 \qquad w\in\Wker_{p,\Omega_a}.
\end{equation*}
In particular,
\begin{equation}\label{eq:zero-correction}
 w_{p,\Omega_a}(U_f)=0,\qquad
 \inf_{w\in\Wker_{p,\Omega_a}}\int_{\Omega_a}|U_f+w|^{p^*}\,dx
 =a\int_D|f|^{p^*}\,dy.
\end{equation}
\end{lemma}
\begin{proof}
By Lemma~\ref{lem:product-exterior} and
Proposition~\ref{thm:potential-representation},
$\Wker_{p,\Omega_a}=\nabla W_0^{1,p^*}(\Omega_a)$.
Let $\varphi\in W_0^{1,p^*}(\Omega_a)$. H\"older's inequality gives
\[
 \int_{\Omega_a}|f(y)|^{p^*-1}|\partial_s\varphi|\dd x
 \leq\|U_f\|_{p^*}^{p^*-1}\|\partial_s\varphi\|_{p^*}<\infty.
\]
For almost every $y\in D$, the slice $\varphi(\cdot,y)$ belongs
to $W_0^{1,p^*}(0,a)$. Hence
\[
 \int_0^a\partial_s\varphi(s,y)\dd s=0.
\]
Using Fubini's theorem, we obtain
\[
 \int_{\Omega_a}|U_f|^{p^*-2}U_f\cdot\nabla\varphi\dd x
 =\int_D|f(y)|^{p^*-2}f(y)
       \left(\int_0^a\partial_s\varphi(s,y)\dd s\right)\dd y
 =0.
\]
Thus zero satisfies \eqref{eq:euler-correction}.
Proposition~\ref{prop:projection} gives
$w_{p,\Omega_a}(U_f)=0$. Finally,
$\int_{\Omega_a}|U_f|^{p^*}\dd x=a\int_D|f|^{p^*}\dd y$,
which proves \eqref{eq:zero-correction}.
\end{proof}

\begin{proof}[Proof of Theorem~\ref{thm:main-intro}]
For $0\ne f\in W_0^{1,p}(D)$, Lemmas~\ref{lem:admissible}
and~\ref{lem:zero-correction} allow $U_f$ in the quotient
\eqref{eq:p-constant}, with zero optimal correction.  By
\eqref{eq:curl-aligned},
\begin{align*}
 \int_{\Omega_a}|\curl U_f|^p\,dx=a\int_D|\nabla f|^p\,dy,
 \qquad
 \int_{\Omega_a}|U_f|^{p^*}\,dx=a\int_D|f|^{p^*}\,dy.
\end{align*}
Consequently,
\begin{align}
 S_{p,\mathrm{curl}}^{\mathrm{int}}(\Omega_a)
 &\leq
 \frac{a\displaystyle\int_D|\nabla f|^p\,\dd y}
 {\left(a\displaystyle\int_D|f|^{p^*}\,\dd y\right)^{p/p^*}}\notag\\
 &=a^{1-p/p^*}
 \frac{\displaystyle\int_D|\nabla f|^p\,\dd y}
 {\left(\displaystyle\int_D|f|^{p^*}\,\dd y\right)^{p/p^*}}. \label{eq:p-bound-calculation}
\end{align}
Since
\begin{equation*}
 1-\frac{p}{p^*}
 =1-\frac{3-p}{3}
 =\frac p3,
\end{equation*}
\eqref{eq:p-bound-calculation} is precisely
\begin{equation}\label{eq:p-bound}
 S_{p,\mathrm{curl}}^{\mathrm{int}}(\Omega_a)
 \leq a^{p/3}
 \frac{\displaystyle\int_D|\nabla f|^p\,dy}
 {\left(\displaystyle\int_D|f|^{p^*}\,dy\right)^{p/p^*}}.
\end{equation}
Taking the infimum over $0\ne f\in W_0^{1,p}(D)$ gives
\begin{equation*}
 S_{p,\mathrm{curl}}^{\mathrm{int}}(\Omega_a)
 \leq a^{p/3}\Lambda_{p,p^*}(D),
\end{equation*}
which proves (i).  For $D=Q$,
\begin{equation*}
 0\leq S_{p,\mathrm{curl}}^{\mathrm{int}}((0,a)\times Q)
 \leq a^{p/3}\Lambda_{p,p^*}(Q)
 \xrightarrow[a\to0]{}0.
\end{equation*}
Since every $(0,a)\times Q$ is bounded, convex, and Lipschitz,
\begin{equation*}
 0\leq
 \inf_{\substack{\Omega\subset\R^3\ \mathrm{bounded,\ convex,}\\
                         \mathrm{Lipschitz}}}
 S_{p,\mathrm{curl}}^{\mathrm{int}}(\Omega)
 \leq
 \inf_{a>0}S_{p,\mathrm{curl}}^{\mathrm{int}}((0,a)\times Q)
 =0,
\end{equation*}
which proves (ii).
\end{proof}

\subsection{Explicit counterexamples and proof of Theorem 2}\label{sec:examples}

We now set $p=2$, so $p^*=6$ and
$S_{2,\mathrm{curl}}^{\mathrm{int}}=\Sint$.

\begin{proof}[Proof of Theorem~\ref{thm:examples-extension}]
Let $f_Q(y_1,y_2)=\sin(\pi y_1)\sin(\pi y_2)$ on $Q=(0,1)^2$.
The elementary integrals give
\begin{equation*}
 \int_Q|\nabla f_Q|^2\,dy=\frac{\pi^2}{2},
 \qquad \int_Q|f_Q|^6\,dy=\frac{25}{256}.
\end{equation*}
Applying \eqref{eq:p-bound} with $p=2$ and $a=1/4$ yields
\[
 \Sint(\Omega_B)\leq
 (1/4)^{2/3}\frac{\pi^2/2}{(25/256)^{1/3}}
 =\pi^2(2/25)^{1/3}.
\]
The zero extension of $f_Q$ to $D_L$ belongs to $H_0^1(D_L)$:
if $\{f_n\}\subset C_c^\infty(Q)$ and $f_n\to f_Q$ in $H_0^1(Q)$,
then the sequence of zero extensions
$\{\widetilde f_n\}\subset C_c^\infty(D_L)$ satisfies
\begin{equation*}
 \widetilde f_n\longrightarrow\widetilde f_Q
 \quad\text{in }H_0^1(D_L),
 \qquad
 \int_{D_L}|\nabla\widetilde f_Q|^2
 =\int_Q|\nabla f_Q|^2,
 \qquad
 \int_{D_L}|\widetilde f_Q|^6
 =\int_Q|f_Q|^6.
\end{equation*}
Thus the same upper bound holds for $\Omega_L$. Finally,
\[
 \left(\frac{\pi^2(2/25)^{1/3}}{3(\pi/2)^{4/3}}\right)^3
 =\frac{32\pi^2}{675}<1.
\]
The upper bound is approximately $4.25268$, while $S\approx5.47790$.

Both domains are bounded Lipschitz prisms. The box $\Omega_B$ is convex,
while the L-shaped cross-section makes $\Omega_L$ nonconvex. At a lateral edge
of $\Omega_B$, the normals of adjacent faces differ, so its boundary is
not $C^1$. The reentrant edge generated by $(1,1)\in\partial D_L$
likewise prevents $\partial\Omega_L$ from being $C^1$.
Thus neither boundary is $C^{1,1}$.
\end{proof}

We now compare the intrinsic constant with the constant obtained
by allowing curl-free corrections in the whole space.

Let \(\widetilde u\) denote the zero extension of
\(u\in\Xzero(\Omega)\) to \(\R^3\), and put
\[
\Wker_{\R^3}:=
 \{w\in\Xzero(\R^3):\curl w=0\}.
\]
\begingroup\color{black}
By the definition of $X_{2,0}(\Omega)$, zero extension gives a
continuous map
\[
 E_0:X_{2,0}(\Omega)\longrightarrow
 W^{6,2}(\curl;\mathbb R^3),\qquad
 \curl(E_0u)=E_0(\curl u)
 \quad\hbox{in }\mathcal D'(\mathbb R^3).
\]
This follows by approximation with compactly supported smooth
fields. In general, the extension need not belong to
$W^{1,2}(\R^3;\R^3)$.
\endgroup
The whole-space-corrected constant is
\begin{equation*}
 \Sext(\Omega)
 :=
 \inf_{\substack{u\in\Xzero(\Omega)\\\curl u\neq0}}
 \frac{\displaystyle\int_\Omega|\curl u|^2\,dx}
 {\left(
   \displaystyle\inf_{w\in\Wker_{\R^3}}
   \int_{\R^3}|\widetilde u+w|^6\,dx
 \right)^{1/3}}.
\end{equation*}

\Needspace{7\baselineskip}
For $\Sext(\Omega)$, the following result was proved in
\cite[Theorems~1.1 and~1.2(a)]{MederskiSzulkin2021}.
\begin{proposition}
\label{thm:external-known}
For every nonempty Lipschitz domain \(\Omega\subset\R^3\),
\begin{equation*}
 \Sext(\Omega)=\Scurl>S.
\end{equation*}
\end{proposition}

\Needspace{7\baselineskip}
The two constants satisfy the following comparison.
\begin{proposition}\label{prop:comparison}
For every bounded Lipschitz domain \(\Omega\),
\begin{equation*}
 \Sint(\Omega)\leq\Sext(\Omega).
\end{equation*}
\end{proposition}

\begin{proof}
Let $w\in\Wker_{2,\Omega}$ and choose a sequence
$\{w_n\}\subset C_c^\infty(\Omega;\R^3)$ such that
\begin{equation*}
 w_n\to w\quad\text{in }L^6(\Omega),
 \qquad
 \curl w_n\to0\quad\text{in }L^2(\Omega).
\end{equation*}
For their zero extensions,
\begin{equation*}
 \widetilde w_n\to\widetilde w\quad\text{in }L^6(\R^3),
 \qquad
 \curl\widetilde w_n
 =\widetilde{\curl w_n}\to0\quad\text{in }L^2(\R^3).
\end{equation*}
Hence
\begin{equation*}
 \widetilde w\in\Wker_{\R^3}.
\end{equation*}
Hence, for every admissible \(u\),
\[
\begin{aligned}
\inf_{\eta\in\Wker_{\mathbb R^3}}
\int_{\mathbb R^3}|\widetilde u+\eta|^6\,dx
&\leq
\inf_{w\in\Wker_{2,\Omega}}
\int_{\mathbb R^3}|\widetilde u+\widetilde w|^6\,dx\\
&=
\inf_{w\in\Wker_{2,\Omega}}
\int_\Omega|u+w|^6\,dx.
\end{aligned}
\]
Therefore, by the definitions of \(\Sint(\Omega)\) and
\(\Sext(\Omega)\),
\[
\Sint(\Omega)\leq\Sext(\Omega).
\]
\end{proof}

In particular, the small values of $\Sint(\Omega)$ in
Theorem~\ref{thm:examples-extension} are consistent with
Proposition~\ref{thm:external-known}.

\begin{remark}
\label{rem:published-bound}
The constant in \eqref{eq:Sint} is exactly the constant
$\overline S_{\mathrm{curl}}(\Omega)$ in
\cite[equation~(1.10)]{MederskiSzulkin2021}.
Theorem~\ref{thm:examples-extension} therefore contradicts the convex-domain lower
bound in their Theorem~1.3. The last paragraph of its proof in Section~5
refers to the proof of Lemma~4.1 there, with the integration domain changed
from $\R^3$ to $\Omega$. More precisely, the first inequality in
\cite[equation~(4.6)]{MederskiSzulkin2021} applies the sharp scalar
Sobolev inequality componentwise to a whole-space field. On $\Omega$,
zero tangential trace does not imply that its components belong to
$H_0^1(\Omega)$, or that their zero extensions belong to
$\mathcal D^{1,2}(\R^3)$. Thus that scalar step, with the same constant $S$,
does not follow from the boundary condition used here.

To see the obstruction directly, let
$U_f(s,y)=f(y)e_1$ on $(0,a)\times D$, with $0\ne f\in H_0^1(D)$.
As proved in Lemmas~\ref{lem:admissible} and~\ref{lem:zero-correction},
$U_f\in X_{2,0}(\Omega_a)$ and its optimal intrinsic correction is zero.
It also satisfies $\dive U_f=0$ in $\Omega_a$ and
\[
 \int_{\Omega_a}|\nabla U_f|^2
 =\int_{\Omega_a}|\curl U_f|^2
 =a\int_D|\nabla f|^2.
\]
\[
 \left.(U_f\cdot\nu)\right|_{\{0\}\times D}=-f,
 \qquad
 \left.(U_f\cdot\nu)\right|_{\{a\}\times D}=f,
 \qquad f\not\equiv0.
\]
The zero extension $\widetilde U_f$ has distributional derivative
\[
 \partial_s\widetilde U_f
 =\bigl(\delta_{s=0}-\delta_{s=a}\bigr)\widetilde f(y)e_1,
\]
where $\widetilde f$ is the zero extension from $D$.
This surface distribution is not an $L^2$ function. Hence
$\widetilde U_f\notin\mathcal D^{1,2}(\R^3,\R^3)$, despite the zero
tangential trace and the domain $H^1$ regularity of $U_f$.
Indeed, the scalar-type quotient of this field is
\[
 \frac{\int_{\Omega_a}|\nabla U_f|^2}
 {\bigl(\int_{\Omega_a}|U_f|^6\bigr)^{1/3}}
 =a^{2/3}\frac{\int_D|\nabla f|^2}
 {\bigl(\int_D|f|^6\bigr)^{1/3}}\longrightarrow0.
\]
Thus that step cannot yield the asserted domain-independent intrinsic
lower bound. This correction leaves the whole-space-corrected result
in Proposition~\ref{thm:external-known} unchanged.
\end{remark}

For any fixed cross-section, the same argument gives an explicit
condition on the thickness which ensures that the intrinsic constant
is smaller than $S$.

\Needspace{7\baselineskip}
\begin{proposition}\label{prop:threshold}
Let \(D\subset\mathbb R^2\) be a bounded Lipschitz domain and choose any
\(0\neq f\in H_0^1(D)\).  Put
\[
 Q_6(f;D):=
 \frac{\displaystyle\int_D|\nabla f|^2\,dy}
      {\left(\displaystyle\int_D|f|^6\,dy\right)^{1/3}}.
\]
Then every thickness
\[
 0<a<\left(\frac{S}{Q_6(f;D)}\right)^{3/2}
\]
satisfies \(\Sint((0,a)\times D)<S\).  Consequently, for every fixed
bounded Lipschitz cross-section \(D\), all sufficiently thin finite
cylinders over \(D\) violate the proposed intrinsic lower bound \(S\).
\end{proposition}
\begin{proof}
By \eqref{eq:p-bound} with $p=2$,
\[
 \Sint((0,a)\times D)\leq a^{2/3}Q_6(f;D)<S,
\]
where the last inequality follows from the assumed bound on $a$.
\end{proof}

\Needspace{7\baselineskip}
The dependence on the thickness is a dependence on the shape of the
domain. Indeed, the intrinsic constant is invariant under the
following transformations.
\begin{proposition}\label{prop:scaling}
If \(R\) is orthogonal, \(b\in\mathbb R^3\), and \(t>0\), then
\[
 \Sint\bigl(b+tR\Omega\bigr)=\Sint(\Omega).
\]
In particular, uniform dilation leaves the intrinsic constant unchanged.
\end{proposition}
\begin{proof}
Set
\begin{equation*}
 \Omega':=b+tR\Omega,
 \qquad
 \xi:=\frac{R^T(x-b)}{t},
 \qquad
 (\mathcal T u)(x):=Ru(\xi).
\end{equation*}
For every orthogonal $R$,
\begin{equation*}
 \curl(\mathcal T u)(x)
 =\frac{\det R}{t}R(\curl u)(\xi),
 \qquad
 \dd x=t^3\,\dd\xi.
\end{equation*}
Thus $\mathcal T$ gives bijections
\begin{equation*}
 \mathcal T:\Xzero(\Omega)\longrightarrow\Xzero(\Omega'),
 \qquad
 \mathcal T:\Wker_{2,\Omega}\longrightarrow\Wker_{2,\Omega'}.
\end{equation*}
For $z=\mathcal T w$,
\begin{align}
 \int_{\Omega'}|\curl(\mathcal T u)|^2\,\dd x
 &=t\int_\Omega|\curl u|^2\,\dd\xi,\notag\\
 \left(\int_{\Omega'}|\mathcal T u+z|^6\,\dd x\right)^{1/3}
 &=t\left(\int_\Omega|u+w|^6\,\dd\xi\right)^{1/3}. \label{eq:rigid-scaling-quotient}
\end{align}
Taking first the infimum over $w$ and then over $u$ in
\eqref{eq:rigid-scaling-quotient} gives
\begin{equation*}
 \Sint(\Omega')\leq\Sint(\Omega).
\end{equation*}
The coordinate map $T(\xi):=b+tR\xi$ has inverse
\begin{equation*}
 T^{-1}(x)=\frac1tR^T(x-b),
 \qquad
 T^{-1}(\Omega')=\Omega.
\end{equation*}
Applying the preceding calculation to $T^{-1}:\Omega'\to\Omega$ gives
\begin{equation*}
 \Sint(\Omega)\leq\Sint(\Omega').
\end{equation*}
Together with the opposite inequality, this proves
\begin{equation*}
 \Sint(\Omega')=\Sint(\Omega).
\end{equation*}
\end{proof}

\section{Proof of Theorem 3}\label{sec:proof-two}

In this section, we prove Theorem~\ref{thm:main-conclusions}.
We first recall the analytic facts needed for the Helmholtz
decomposition and Newtonian representation. We then exclude two
special forms of solutions.

\subsection{Analytic preliminaries}\label{sec:preliminaries}

We collect the standard analytic facts used below. All differential
identities are understood distributionally. Define the curl-free space
\begin{equation*}
 \G_6:=\{g\in W^{6,2}(\curl;\R^3):\curl g=0\}.
\end{equation*}

\Needspace{7\baselineskip}
We first recall the $L^p$ Helmholtz decomposition.
\begin{lemma}
\label{lem:helmholtz-lp}
For \(1<p<\infty\), there is a bounded projection \(P\) on
\(L^p(\R^3,\R^3)\) such that every \(u\in L^p\) has the unique decomposition
\begin{equation}\label{eq:hodge}
 u=v+g,\qquad v=Pu,\qquad g=(I-P)u,\qquad
 \dive v=0,\qquad \curl g=0.
\end{equation}
Moreover, \(\norm v_p+\norm g_p\leq C_p\norm u_p\).
For \(p=6\), if \(u\in W^{6,2}(\curl;\R^3)\), then \(g\in\G_6\) and
\(g=\nabla\varphi\) for some \(\varphi\in W_{\mathrm{loc}}^{1,6}(\R^3)\),
unique up to a constant.
\end{lemma}

\begin{proof}
For $\xi\neq0$, set
\begin{equation*}
 P_{ij}(\xi):=\delta_{ij}-\frac{\xi_i\xi_j}{|\xi|^2}.
\end{equation*}
Since $P_{ij}=\delta_{ij}+R_iR_j$, boundedness of the Riesz transforms
on $L^p$, $1<p<\infty$, gives
\begin{equation*}
 \|Pu\|_p+\|(I-P)u\|_p\leq C_p\|u\|_p.
\end{equation*}
The symbol identities
\begin{equation*}
 P(\xi)^2=P(\xi),
 \qquad
 \xi\cdot P(\xi)a=0,
 \qquad
 \xi\times(I-P(\xi))a=0
\end{equation*}
imply, for
\begin{equation*}
 v:=Pu,
 \qquad
 g:=(I-P)u,
\end{equation*}
that
\begin{equation*}
 u=v+g,
 \qquad
 \dive v=0,
 \qquad
 \curl g=0
\end{equation*}
in distributions.

If $u=v_1+g_1=v_2+g_2$ are two such decompositions, then
\begin{equation*}
 h:=v_1-v_2=g_2-g_1,
 \qquad
 \dive h=0,
 \qquad
 \curl h=0.
\end{equation*}
Hence
\begin{equation*}
 \Delta h
 =\nabla(\dive h)-\curl(\curl h)
 =0,
 \qquad h\in L^p(\R^3),
\end{equation*}
and the mean-value estimate gives $h=0$.  This proves uniqueness.
For $p=6$, $g\in L^6$ and $\curl g=0$, so
\begin{equation*}
 g\in\G_6.
\end{equation*}
The distributional Poincar\'e lemma yields
\begin{equation*}
 g=\nabla\varphi,
 \qquad
 \varphi\in W_{\mathrm{loc}}^{1,6}(\R^3),
\end{equation*}
with $\varphi$ unique modulo constants.
\end{proof}

\Needspace{7\baselineskip}
The curl-free space has the following closedness property.
\begin{lemma}\label{lem:G6-closed}
The space \(\G_6\) is a closed linear subspace of \(L^6(\R^3,\R^3)\).
On \(\G_6\), the \(L^6\) and \(W^{6,2}(\curl;\R^3)\) graph norms coincide.
\end{lemma}

\begin{proof}
If \(\{g_n\}\subset\G_6\) and \(g_n\to g\) in \(L^6\), testing against
\(\curl\phi\), \(\phi\in C_c^\infty(\R^3,\R^3)\), gives
\begin{align*}
 \langle\curl g,\phi\rangle
 &=\int_{\R^3}g\cdot\curl\phi\,\dd x\\
 &=\lim_{n\to\infty}
   \int_{\R^3}g_n\cdot\curl\phi\,\dd x
 =\lim_{n\to\infty}\langle\curl g_n,\phi\rangle
 =0.
\end{align*}
Thus $g\in\G_6$.  Moreover, for every $g\in\G_6$,
\begin{equation*}
 \|g\|_{W^{6,2}(\curl;\R^3)}
 =\|g\|_6+\|\curl g\|_2
 =\|g\|_6.
\end{equation*}
\end{proof}

\Needspace{7\baselineskip}
The following identity reduces the equation for the solenoidal
component to a Poisson equation.
\begin{lemma}
\label{lem:curl-curl-identity}
For every vector-valued distribution \(w\) on \(\R^3\),
\begin{equation*}
 \curl(\curl w)=\nabla(\dive w)-\Delta w.
\end{equation*}
Consequently, if \(\dive v=0\) and \(\curl g=0\), then
\(\curl(\curl(v+g))=-\Delta v\).
\end{lemma}

\begin{proof}
Componentwise, contraction of the Levi--Civita symbols gives
\begin{align*}
 [\curl(\curl w)]_i
 &=\eps_{ijk}\partial_j
   (\eps_{k\ell m}\partial_\ell w_m)\\
 &=(\delta_{i\ell}\delta_{jm}-\delta_{im}\delta_{j\ell})
   \partial_j\partial_\ell w_m\\
 &=\partial_i\partial_jw_j-\partial_j\partial_jw_i
 =[\nabla(\dive w)-\Delta w]_i.
\end{align*}
Distributional derivatives commute, so the calculation is valid in
$\mathcal D'(\R^3)$.  Therefore, if $w=v+g$ with
\begin{equation*}
 \dive v=0,
 \qquad
 \curl g=0,
\end{equation*}
then
\begin{equation*}
 \curl(\curl(v+g))
 =\curl(\curl v)
 =-\Delta v.
\end{equation*}
\end{proof}

\Needspace{7\baselineskip}
We shall use the following interior estimate for the Poisson equation.
\begin{lemma}
\label{lem:local-poisson}
Let \(1<p<\infty\), let \(f,w\in L_{\mathrm{loc}}^p(\R^3)\), and suppose
that \(-\Delta w=f\). For concentric balls \(B_r\Subset B_R\),
\begin{equation*}
 \norm{w}_{W^{2,p}(B_r)}
 \leq C_{p,r,R}
 \bigl(\norm f_{L^p(B_R)}+\norm w_{L^p(B_R)}\bigr).
\end{equation*}
In particular,
\(W_{\mathrm{loc}}^{2,6/5}\hookrightarrow
W_{\mathrm{loc}}^{1,2}\hookrightarrow L_{\mathrm{loc}}^6\).
\end{lemma}

\begin{proof}
The estimate is the interior Calder\'on--Zygmund estimate
\cite{Evans2010,Stein1970}. The stated embeddings follow from
the Sobolev embedding theorem, since
$1/(6/5)-1/3=1/2$ and $1/2-1/3=1/6$.
\end{proof}

\Needspace{7\baselineskip}
For estimates on the whole space, we recall the
Hardy--Littlewood--Sobolev inequality.
\begin{lemma}\label{lem:HLS}
Let \(n\geq1\), \(0<\alpha<n\), \(1<p<n/\alpha\), and
\(1/q=1/p-\alpha/n\), equivalently \(q=np/(n-\alpha p)\).
Then the Riesz potential satisfies
\begin{align}
 I_\alpha f(x)
 &:=\int_{\R^n}\frac{f(y)}{|x-y|^{n-\alpha}}\dd y,
 \notag\\
 \norm{I_\alpha f}_{L^q(\R^n)}
 &\leq C(n,\alpha,p)\norm f_{L^p(\R^n)}.
 \label{eq:HLS-inequality}
\end{align}
The estimate also holds for vector-valued \(f\), with Euclidean norms.
\end{lemma}

\begin{proof}
For scalar $f$, \eqref{eq:HLS-inequality} is the classical
Hardy--Littlewood--Sobolev estimate \cite{LiebLoss,Stein1970}.  If
$f=(f_1,\dots,f_m)$, then pointwise
\begin{align*}
 |I_\alpha f(x)|
 &=\left|
   \int_{\R^n}\frac{f(y)}{|x-y|^{n-\alpha}}\,\dd y
   \right|\\
 &\leq
 \int_{\R^n}\frac{|f(y)|}{|x-y|^{n-\alpha}}\,\dd y
 =I_\alpha(|f|)(x).
\end{align*}
Therefore,
\begin{equation*}
 \|I_\alpha f\|_q
 \leq\|I_\alpha(|f|)\|_q
 \leq C(n,\alpha,p)\|f\|_p.
\end{equation*}
\end{proof}

\Needspace{7\baselineskip}
Applying Lemma~\ref{lem:HLS} to the Newtonian kernel and its
gradient gives the estimates below.
\begin{lemma}\label{lem:newton-estimates}
Let \(N(x)=(4\pi|x|)^{-1}\). If \(F\in L^{6/5}(\R^3,\R^3)\), then
\(V=N*F\) satisfies \(-\Delta V=F\) and
\begin{equation*}
 \norm V_6+\norm{\nabla V}_2\leq C\norm F_{6/5}.
\end{equation*}
Consequently, \(\norm{\curl V}_2\leq C\norm F_{6/5}\).
\end{lemma}

\begin{proof}
Since
\begin{equation*}
 V=\frac1{4\pi}I_2F,
 \qquad
 |\nabla V|\leq CI_1(|F|),
\end{equation*}
Lemma~\ref{lem:HLS}, with $p=6/5$, gives
\begin{align*}
 \frac16&=\frac{5}{6}-\frac23
 &&\Longrightarrow&
 \|V\|_6&\leq C\|F\|_{6/5},\\
 \frac12&=\frac{5}{6}-\frac13
 &&\Longrightarrow&
 \|\nabla V\|_2&\leq C\|F\|_{6/5}.
\end{align*}
For a sequence $\{F_n\}\subset C_c^\infty$ with
$F_n\to F$ in $L^{6/5}$, define the sequence $\{V_n\}$ by
\begin{equation*}
 V_n:=N*F_n.
\end{equation*}
The preceding estimates imply
\begin{equation*}
 V_n\longrightarrow V\quad\text{in }L^6,
 \qquad
 \nabla V_n\longrightarrow\nabla V\quad\text{in }L^2.
\end{equation*}
Since $-\Delta V_n=F_n$, passage to distributions yields
\begin{equation*}
 -\Delta V=F.
\end{equation*}
Finally,
\begin{equation*}
 \|\curl V\|_2\leq C\|\nabla V\|_2
 \leq C\|F\|_{6/5}.
\end{equation*}
\end{proof}

\Needspace{7\baselineskip}
To identify the Newtonian potential with the solenoidal component,
we use the following uniqueness result.
\begin{lemma}
\label{lem:harmonic-Lp}
If \(1\leq p<\infty\), \(w\in L^p(\R^3,\R^m)\), and \(\Delta w=0\),
then \(w=0\).
\end{lemma}

\begin{proof}
By Weyl's lemma, every component $w_i$ is smooth and harmonic.  For
$x\in\R^3$ and $R>0$, the mean-value property and H\"older's inequality
give
\begin{align*}
 |w_i(x)|
 &\leq\frac1{|B_R|}\int_{B_R(x)}|w_i(y)|\,\dd y\\
 &\leq |B_R|^{-1/p}\|w_i\|_{L^p(B_R(x))}\\
 &\leq C R^{-3/p}\|w_i\|_{L^p(\R^3)}.
\end{align*}
Letting $R\to\infty$ gives $w_i(x)=0$ for every $i$ and $x$.
\end{proof}

\Needspace{7\baselineskip}
We also need a density result to use finite-energy fields as test
functions in the weak equation.
\begin{lemma}\label{lem:graph-density}
The space \(C_c^\infty(\R^3,\R^3)\) is dense in
\(W^{6,2}(\curl;\R^3)\) with respect to its graph norm.
\end{lemma}

\begin{proof}
Choose \(\chi\in C_c^\infty(B_2)\) with \(0\leq\chi\leq1\) and
\(\chi=1\) on \(B_1\). Set \(\chi_R(x)=\chi(x/R)\) and
\(A_R=B_{2R}\setminus B_R\). For \(u\in W^{6,2}(\curl;\R^3)\),
\[
 \curl(\chi_Ru)=\chi_R\curl u+\nabla\chi_R\times u.
\]
By H\"older's inequality and \(|A_R|\leq CR^3\),
\[
 \norm{\nabla\chi_R\times u}_2
 \leq \frac{C}{R}|A_R|^{1/3}\norm u_{L^6(A_R)}
 \leq C\norm u_{L^6(A_R)}\longrightarrow0.
\]
The remaining field and curl tails tend to zero in \(L^6\) and
\(L^2\), respectively. Thus \(\chi_Ru\to u\) in the graph norm.
For fixed \(R\), mollification commutes with curl and converges
in both norms. A diagonal choice gives the required compactly
supported smooth approximations.
\end{proof}

\subsection{Nonlinear Hodge decomposition and reconstruction}\label{sec:hodge-system}

For \(u\in W^{6,2}(\curl;\R^3)\), apply Lemma~\ref{lem:helmholtz-lp} with \(p=6\)
and write
\[
    u=v+g=v+\nabla\varphi,
    \qquad \dive v=0,
    \qquad g\in\G_{6}.
\]

\Needspace{7\baselineskip}
For a solution of \eqref{eq:main}, the two components satisfy
the following system.
\begin{proposition}\label{prop:coupled}
Let \(u\in W^{6,2}(\curl;\R^3)\) be a weak solution of \eqref{eq:main}, and let
\(u=v+\nabla\varphi\) be the decomposition \eqref{eq:hodge}.  Then
\begin{align}
    &-\Delta v=\abs{u}^{4}u,
        \label{eq:poisson-v}\\
    &\dive(\abs{u}^{4}u)=0,
        \label{eq:divF}\\
    &\dive\!\left(
       \abs{v+\nabla\varphi}^{4}(v+\nabla\varphi)
    \right)=0
        \label{eq:gauge}
\end{align}
in the sense of distributions.  Moreover,
\begin{equation}\label{eq:local-reg}
    v\in W_{\mathrm{loc}}^{2,6/5}(\R^{3},\R^{3}),
    \qquad
    \curl  u=\curl  v
    \in W_{\mathrm{loc}}^{1,6/5}(\R^{3},\R^{3}).
\end{equation}
\end{proposition}

\begin{proof}
Set \(F:=\abs{u}^{4}u\).  Since \(u\in L^{6}\),
\begin{equation}\label{eq:F-integrability}
    \int_{\R^{3}}\abs{F}^{6/5}\dd x
    =\int_{\R^{3}}\abs{u}^{6}\dd x<\infty,
    \qquad
    \norm{F}_{6/5}=\norm{u}_{6}^{5}.
\end{equation}
Thus every pairing below is well defined.

First take \(\phi=\nabla\psi\) in \eqref{eq:weak}, where
\(\psi\in C_{c}^{\infty}(\R^{3})\).  This is an admissible vector-valued
test function, and
\(\curl(\nabla\psi)=0\).  Hence the left-hand side of
\eqref{eq:weak} is zero and
\[
    \int_{\R^{3}}F\cdot\nabla\psi\dd x=0
    \qquad(\psi\in C_{c}^{\infty}(\R^{3})).
\]
By the definition
\(\langle\dive F,\psi\rangle=-\int F\cdot\nabla\psi\), this is precisely
\(\dive F=0\), proving \eqref{eq:divF}.

Next, Lemma~\ref{lem:helmholtz-lp} gives
\(u=v+\nabla\varphi\), \(\dive v=0\), and
\(\curl \nabla\varphi=0\).  Lemma~\ref{lem:curl-curl-identity}
therefore yields
\begin{equation}\label{eq:curlcurl-to-poisson}
    \curl(\curl  u)=-\Delta v
    \qquad\text{in }\mathcal D'(\R^{3},\R^{3}).
\end{equation}
Comparing \eqref{eq:curlcurl-to-poisson} with the equation
\(\curl(\curl  u)=F\) proves \eqref{eq:poisson-v}.
Substituting \(u=v+\nabla\varphi\) into \(\dive F=0\) gives
\eqref{eq:gauge}. No condition \(\dive u=0\) has been used.

Finally, \(v\in L^{6}\) implies \(v\in L^{6/5}_{\mathrm{loc}}\) by
H\"older's inequality on bounded balls:
\begin{equation*}
 \|v\|_{L^{6/5}(B_R)}
 \leq |B_R|^{2/3}\|v\|_{L^6(B_R)}.
\end{equation*}
Hence \eqref{eq:F-integrability}, \eqref{eq:poisson-v}, and
Lemma~\ref{lem:local-poisson} imply
\begin{equation*}
 v\in W_{\mathrm{loc}}^{2,6/5}(\R^3;\R^3).
\end{equation*}
Furthermore,
\begin{equation*}
 \curl u=\curl v,
 \qquad
 \partial_\ell(\curl v)_i
 =\eps_{ijk}\partial_\ell\partial_jv_k
 \in L_{\mathrm{loc}}^{6/5},
\end{equation*}
and therefore
\begin{equation*}
 \curl u\in W_{\mathrm{loc}}^{1,6/5}(\R^3;\R^3).
\end{equation*}
This proves \eqref{eq:local-reg}.
\end{proof}

We next derive a representation of the solution. Recall the
Newtonian kernel
\begin{equation*}
    N(x):=\frac{1}{4\pi\abs{x}}.
\end{equation*}

For \(v\in L^{6}(\R^{3},\R^{3})\), let \(\Gamma(v)\in\G_{6}\) denote the
minimizer of
\begin{equation}\label{eq:gauge-min}
    g\longmapsto \int_{\R^{3}}\abs{v+g}^{6}\dd x,
    \qquad g\in\G_{6}.
\end{equation}

\Needspace{7\baselineskip}
The existence and uniqueness of this correction follow from
Lemma~\ref{lem:strict-convex-projection}.
\begin{lemma}\label{lem:Gamma}
For every \(v\in L^{6}\), the minimizer \(\Gamma(v)\) in
\eqref{eq:gauge-min} exists and is unique.  It is characterized by
\begin{equation}\label{eq:Gamma-Euler}
    \int_{\R^{3}}
       \abs{v+\Gamma(v)}^{4}(v+\Gamma(v))\cdot h\dd x=0
    \qquad\text{for every }h\in\G_{6}.
\end{equation}
It also satisfies
\begin{equation*}
    \norm{v+\Gamma(v)}_{6}\leq\norm{v}_{6},
    \qquad
    \norm{\Gamma(v)}_{6}\leq2\norm{v}_{6}.
\end{equation*}
\end{lemma}

\begin{proof}
By Lemma~\ref{lem:G6-closed}, $\G_6$ is a closed linear subspace
of $L^6$. Apply Lemma~\ref{lem:strict-convex-projection} with
$E=\R^3$, $p=6$, $Z=\G_6$, and $q=v$. It gives the unique
minimizer, the Euler identity \eqref{eq:Gamma-Euler}, and both
norm estimates.
\end{proof}

\Needspace{7\baselineskip}
We can now express the solution in terms of its nonlinear source
and the correction $\Gamma$.
\begin{proposition}\label{thm:reconstruction}
Let \(u\in W^{6,2}(\curl;\R^3)\) solve \eqref{eq:main} and set
\begin{equation*}
    F:=\abs{u}^{4}u.
\end{equation*}
Then its solenoidal component is
\begin{equation}\label{eq:v-NF}
    v=N*F=(-\Delta)^{-1}F
    \qquad\text{in }L^{6}(\R^{3},\R^{3}),
\end{equation}
and
\begin{equation}\label{eq:curl-NF}
    \curl  u=\curl  v=\curl(N*F)
    \qquad\text{in }L^{2}(\R^{3},\R^{3}).
\end{equation}
Furthermore,
\begin{equation}\label{eq:HLS-est}
    \norm{v}_{6}\leq C\norm{F}_{6/5},
    \qquad
    \norm{\curl  u}_{2}\leq C\norm{F}_{6/5},
\end{equation}
and the full field satisfies the exact implicit formula
\begin{equation}\label{eq:full-reconstruction}
    u=N*F+\Gamma(N*F),
    \qquad
    F=\abs{u}^{4}u.
\end{equation}
\end{proposition}

\begin{proof}
Because \(u\in L^{6}\), identity \eqref{eq:F-integrability} gives
\(F\in L^{6/5}\).  Lemma~\ref{lem:newton-estimates} therefore applies to
\[
    V:=N*F
\]
and shows that
\begin{equation*}
    V\in L^{6},
    \qquad -\Delta V=F,
    \qquad
    \norm{V}_{6}+\norm{\nabla V}_{2}
       \leq C\norm{F}_{6/5}.
\end{equation*}
On the other hand, Proposition~\ref{prop:coupled} gives
\(-\Delta v=F\), and Lemma~\ref{lem:helmholtz-lp} gives \(v\in L^{6}\).
Thus
\[
    w:=v-V\in L^{6}(\R^{3},\R^{3}),
    \qquad \Delta w=0
\]
in distributions.  Lemma~\ref{lem:harmonic-Lp}, with \(p=6\), yields
\(w=0\).  Hence
\[
    v=V=N*F,
\]
which proves \eqref{eq:v-NF} and the first estimate in
\eqref{eq:HLS-est}.

The Helmholtz decomposition gives
\(u=v+g\) with \(\curl  g=0\).  Consequently
\[
    \curl  u=\curl  v
    =\curl(N*F).
\]
The last estimate in Lemma~\ref{lem:newton-estimates} proves both
\eqref{eq:curl-NF} and the second inequality in \eqref{eq:HLS-est}.

It remains to identify \(g\).  Fix \(h\in\G_{6}\).  By
Lemma~\ref{lem:graph-density}, there is a sequence
\(\{h_{n}\}\subset C_{c}^{\infty}(\R^{3},\R^{3})\) such that
\begin{equation*}
    h_{n}\to h\quad\text{in }L^{6},
    \qquad
    \curl  h_{n}\to\curl  h=0
       \quad\text{in }L^{2}.
\end{equation*}
Using \(h_{n}\) in the weak equation gives
\begin{equation*}
 \int_{\R^{3}}F\cdot h_{n}\,\dd x
 =\int_{\R^{3}}(\curl u)\cdot(\curl h_{n})\,\dd x.
\end{equation*}
The two limits are quantified by
\begin{align*}
 \left|
 \int_{\R^3}(\curl u)\cdot(\curl h_n)\,\dd x
 \right|
 &\leq\|\curl u\|_2\|\curl h_n\|_2
 \longrightarrow0,\\
 \left|
 \int_{\R^3}F\cdot(h_n-h)\,\dd x
 \right|
 &\leq\|F\|_{6/5}\|h_n-h\|_6
 \longrightarrow0.
\end{align*}
Therefore
\begin{equation}\label{eq:F-orthogonal-G6}
    \int_{\R^{3}}F\cdot h\dd x=0
    \qquad(h\in\G_{6}).
\end{equation}
Since \(F=\abs{u}^{4}u=\abs{v+g}^{4}(v+g)\), equation
\eqref{eq:F-orthogonal-G6} is exactly the Euler equation
\eqref{eq:Gamma-Euler} with \(g\) in place of \(\Gamma(v)\).
Lemma~\ref{lem:Gamma} gives uniqueness, so
\[
    g=\Gamma(v)=\Gamma(N*F).
\]
Substituting this identity into \(u=v+g\) proves
\eqref{eq:full-reconstruction}.
\end{proof}

\begin{remark}\label{rem:implicit}
Formula \eqref{eq:full-reconstruction} reduces the reconstruction
of a solution to the Newtonian potential and the correction $\Gamma$.
The source $F$ must satisfy
\[
    F
    =\abs{N*F+\Gamma(N*F)}^{4}
      \bigl(N*F+\Gamma(N*F)\bigr).
\]
Thus the representation is implicit. Ground states correspond to
the least-energy nonzero solutions of this relation and, after
curl-free correction and amplitude normalization, to optimizers of
the sharp Sobolev--curl inequality. We refer to
\cite[Theorem~1.2(b)]{MederskiSzulkin2021}.
\end{remark}

\subsection{Excluded ansatzes}\label{sec:excluded-ansatzes}

We next consider two possible forms of solutions. Since the
Aubin--Talenti functions are radial about their centers, we first
ask whether a constant vector multiple of a radial function can
solve \eqref{eq:main}.

\Needspace{7\baselineskip}

\begin{proposition}
\label{prop:no-talenti}
Let \(a\in\R^{3}\setminus\{0\}\), \(y_{0}\in\R^{3}\), and let
\(U\in C^{1}((0,\infty))\) be nonconstant.  Then
\begin{equation}\label{eq:polarized}
    u(x)=a\,U(\abs{x-y_{0}})
\end{equation}
cannot be a weak solution of \eqref{eq:main}.  In particular, no nonzero
constant vector multiple of an Aubin--Talenti function is a curl--curl ground
state.
\end{proposition}

\begin{proof}
Every weak solution satisfies \eqref{eq:divF}.  For the field
\eqref{eq:polarized}, set \(r=\abs{x-y_{0}}\) and
\(G(r):=\abs{U(r)}^{4}U(r)\).  Since
\[
    \abs{u}^{4}u
    =\abs{a}^{4}a\,G(r),
    \qquad
    \nabla r=\frac{x-y_{0}}{\abs{x-y_{0}}},
\]
the chain rule on \(\R^{3}\setminus\{y_{0}\}\) gives
\[
    \dive(\abs{u}^{4}u)
    =\abs{a}^{4}G'(\abs{x-y_{0}})
       \frac{a\cdot(x-y_{0})}{\abs{x-y_{0}}}.
\]
The scalar map \(s\mapsto\abs{s}^{4}s\) is strictly increasing.  Hence the
nonconstant function \(U\) makes \(G\) nonconstant.  Because \(G\in C^{1}\),
the mean-value theorem yields a radius \(r_{0}>0\) with
\(G'(r_{0})\neq0\).  At the point
\(x_{0}=y_{0}+r_{0}a/\abs a\), the displayed divergence equals
\[
    \abs a^{5}G'(r_{0})\neq0.
\]
It is therefore nonzero on a neighborhood of \(x_{0}\), contradicting the
distributional identity \(\dive(\abs u^{4}u)=0\).
\end{proof}

\Needspace{7\baselineskip}
For the second exclusion, we use the following fact about the
Laplacian on the whole space.
\begin{lemma}\label{lem:free-helmholtz-L2}
Let \(\lambda>0\).  If \(w\in L^{2}(\R^{3},\R^{m})\) satisfies
\begin{equation}\label{eq:free-helmholtz-lemma}
    -\Delta w=\lambda w
\end{equation}
in distributions, then \(w=0\).
\end{lemma}

\begin{proof}
We argue componentwise and use the Fourier convention
\begin{equation*}
 \widehat f(\xi)=(2\pi)^{-3/2}
 \int_{\R^3}f(x)e^{-ix\cdot\xi}\dd x.
\end{equation*}
For a smooth rapidly decreasing function $f$, integration by parts
in the $x_j$ variable gives
\begin{align*}
 \widehat{\partial_j f}(\xi)
 &=(2\pi)^{-3/2}\int_{\R^3}\partial_j f(x)e^{-ix\cdot\xi}\dd x\\
 &=-(2\pi)^{-3/2}\int_{\R^3}f(x)\partial_j(e^{-ix\cdot\xi})\dd x\\
 &=i\xi_j\widehat f(\xi).
\end{align*}
The boundary term vanishes by rapid decay. Applying this formula
twice yields
\begin{equation*}
 \widehat{\partial_j^2 f}(\xi)
 =(i\xi_j)^2\widehat f(\xi)
 =-\xi_j^2\widehat f(\xi).
\end{equation*}
Summing over the three coordinates, we obtain
\begin{equation*}
 \widehat{-\Delta f}(\xi)
 =(\xi_1^2+\xi_2^2+\xi_3^2)\widehat f(\xi)
 =|\xi|^2\widehat f(\xi).
\end{equation*}

The same identity holds for tempered distributions, in particular for
$w\in L^2$. The weak formulation of \eqref{eq:free-helmholtz-lemma}
extends to Schwartz test functions by a standard cutoff argument.
Taking the Fourier transform therefore gives
\begin{equation*}
 |\xi|^2\widehat w=\lambda\widehat w,
 \qquad\text{hence}\qquad
 (|\xi|^2-\lambda)\widehat w=0
\end{equation*}
in distributions. Plancherel's theorem gives $\widehat w\in L^2$.

If an open set \(O\) has closure disjoint from the sphere
\(S_{\sqrt\lambda}=\{\abs\xi=\sqrt\lambda\}\), then
\((\abs\xi^{2}-\lambda)^{-1}\in C^\infty(O)\).  For every
$\phi\in C_c^\infty(O)$, put
\begin{equation*}
 \psi(\xi):=\frac{\phi(\xi)}{|\xi|^2-\lambda}
 \in C_c^\infty(O).
\end{equation*}
Then
\begin{equation*}
 \langle\widehat w,\phi\rangle
 =\langle(|\xi|^2-\lambda)\widehat w,\psi\rangle
 =0.
\end{equation*}
Thus $\widehat w$ is the zero distribution on every such $O$.
Because $\widehat w\in L^2$, it vanishes almost everywhere there.
Covering the complement of the sphere by countably many such open
sets gives
\begin{equation*}
 \operatorname{ess\,supp}\widehat w
 \subset S_{\sqrt\lambda}.
\end{equation*}
Since $|S_{\sqrt\lambda}|=0$ and $\widehat w\in L^2$,
\begin{equation*}
 \|\widehat w\|_2^2
 =\int_{S_{\sqrt\lambda}}|\widehat w|^2\,\dd\xi
 =0.
\end{equation*}
Plancherel's theorem now gives $w=0$.
\end{proof}

\Needspace{7\baselineskip}
We now exclude solutions satisfying a constant-coefficient
Beltrami relation.
\begin{proposition}
\label{prop:no-beltrami}
If \(\kappa\neq0\), there is no nonzero field
\(u\in W^{6,2}(\curl;\R^3)\) satisfying
\begin{equation}\label{eq:beltrami}
    \curl  u=\kappa u
\end{equation}
on \(\R^{3}\).  If \(\kappa=0\), no nonzero weak solution of
\eqref{eq:main} can satisfy \eqref{eq:beltrami}.  Consequently, a nonzero
finite-energy solution of \eqref{eq:main} cannot satisfy a
constant-coefficient Beltrami relation for any \(\kappa\in\R\).
\end{proposition}

\begin{proof}
If \(\kappa=0\) and \(u\) also solves \eqref{eq:main}, then
the left-hand side of \eqref{eq:main} is zero.  Hence
\(\abs{u}^{4}u=0\) in distributions and therefore almost everywhere, so
\(u=0\).

Suppose \(\kappa\neq0\).  Since \(\curl  u\in L^{2}\), relation
\eqref{eq:beltrami} implies
\[
    \norm{u}_{2}=\abs\kappa^{-1}\norm{\curl  u}_{2}<\infty.
\]
Taking distributional divergence of \eqref{eq:beltrami} and using
\(\dive(\curl  u)=0\) gives
\(\kappa\dive u=0\), hence \(\dive u=0\).  Taking one more curl and
using Lemma~\ref{lem:curl-curl-identity} yields
\begin{equation*}
    -\Delta u=\kappa^{2}u
\end{equation*}
in distributions.  By Plancherel's theorem,
Lemma~\ref{lem:free-helmholtz-L2} with \(\lambda=\kappa^{2}\) gives
\(u=0\).
\end{proof}

\begin{proof}[Proof of Theorem~\ref{thm:main-conclusions}]
Part~(i) follows from Lemma~\ref{lem:helmholtz-lp} and
Proposition~\ref{thm:reconstruction}. The exclusions
in part~(ii) are proved in Propositions~\ref{prop:no-talenti}
and~\ref{prop:no-beltrami}.
\end{proof}

\

{\bf Acknowledgments:} The research of Yong Liu was supported by the National Natural Science Foundation of China (Grant No. 12471204). The research of J. Wang is partially supported
by National Key RD Program of China 2022YFA1005601 and National Natural Science
Foundation of China 12371114. The research of Wen Yang
was supported by the National Key Research and Development Program of China
(Grant No. 2022YFA1006800), the National Natural Science Foundation of China
(Grant No. 12531010), the Science and Technology Development Fund of the
Macao SAR (FDCT, Grant No. 0070/2024/RIA1), the Multi-Year Research Grants
of the University of Macau (Grant Nos. MYRG-GRG2025-00051-FST and MYRGGRG2026-00034-FSC), and the University of Macau Development Foundation
(Grant No. TISF/2025/006/FST).The authors acknowledge the use of AI tools. 

\

\end{document}